\documentclass[11pt]{amsart}

\usepackage{amsfonts}
\usepackage{amsthm}
\usepackage{amssymb}
\usepackage{amsmath}
\usepackage{url, hyperref}

\newtheorem{theo}{\bf Theorem}[section]
\newtheorem{lemma}{\bf Lemma}[section]

\newtheorem{conj}{\bf Conjecture}[section]

\newcommand{\spn}{{\rm span}}

\newcommand{\Z}{{\mathbb Z}}

\newcommand{\R}{{\mathbb R}}

\newcommand{\bea}{\begin{eqnarray*}}
\newcommand{\eea}{\end{eqnarray*}}
\newcommand{\be}{\begin{eqnarray}}
\newcommand{\ee}{\end{eqnarray}}

\newcommand{\vol}{\mbox{vol}\,}

\newcommand{\ve}{\boldsymbol}

\newcommand{\interior}{\mbox{\rm int}\,}
\newcommand{\frob}{\mathrm{F}}
\newcommand{\dfrob}{\mathrm{g}}

\numberwithin{equation}{section}

\begin{document}

\title[LLL for Integer Knapsacks]{LLL-reduction for Integer Knapsacks}
\author{Iskander Aliev}
\address{School of Mathematics and Wales Institute of Mathematical and Computational Sciences, Cardiff University, Senghennydd Road, Cardiff, Wales, UK}
\email{alievi@cf.ac.uk}

\author{Martin Henk}
\address{Fakult\"at f\"ur Mathematik, Otto-von-Guericke
Universit\"at Mag\-deburg, Universit\"atsplatz 2, D-39106 Magdeburg,
Germany} \email{martin.henk@ovgu.de}


\begin{abstract}
 Given a matrix  $A\in \Z^{m\times n}$ satisfying certain
 regularity assumptions, a well-known integer programming problem asks to find an integer point in the associated {\em knapsack polytope}
\bea
P(A,{\ve b})=\{{\ve x}\in \R^n_{\ge 0}: A {\ve x}={\ve b}\}\,
\eea
or determine that no such point exists.
We obtain an LLL-based polynomial time algorithm that solves the problem subject to a constraint on the location of the vector ${\ve b}$.

\end{abstract}



\maketitle

\section{Introduction and Statement of Results}

Let $A\in\Z^{m\times n}$, $1\leq m<n$, be an integral
$m\times n$ matrix satisfying
\begin{equation}
\begin{split}
{\rm i)}&\,\, \gcd\left(\det(A_{I_m}) : A_{I_m}\text{ is an $m\times
    m$ minor of }A\right)=1, \\
{\rm ii)}&\,\, \{{\ve x}\in\R^n_{\ge 0}: A\,{\ve x}={\ve 0}\}=\{{\ve 0}\},
\end{split}
\label{assumption}
\end{equation}
where $\gcd(a_1,\dots,a_l)$ denotes the greatest common divisor of
integers $a_i$, $1\leq i\leq l$.
For such a matrix $A$ and a vector  ${\ve b}\in \Z^m$  the   {\em knapsack
  polytope} $P(A,{\ve b})$ is defined as
\bea
P(A,{\ve b})=\{{\ve x}\in \R^n_{\ge 0}: A {\ve x}={\ve b}\}\,.
\label{P}
\eea
Observe that on account of \eqref{assumption} ii), $P(A,{\ve b})$ is
indeed a polytope (or empty).

The paper is concerned with the following
integer programming problem:
\be
\begin{array}{l}
\mbox{Given input}\; (A,{\ve b}),\;\mbox{find an integer point in}\; P(A,{\ve b})\;\\ \mbox{or determine that no such  point exists}\,.
\end{array}
\label{Knapsack}
\ee
The problem (\ref{Knapsack}) is  NP-hard (see e.g. Section 15.6 in Papadimitriou and Steiglitz \cite{PS}).
When $m=1$ we obtain the well-known {\em integer knapsack problem}: given integers $a_j$, $j=1,\ldots,n$, and $b$, find
integers $x_j\ge 0$, $j=1,\ldots,n$, such that $\sum_{j=1}^n a_j x_j=b$ or determine that no such integers exist.

Let us define the set
\begin{equation*}
{\mathcal F}(A)=\{{\ve b}\in\Z^m : P(A,{\ve b})\cap\Z^n\ne\emptyset\}.
\end{equation*}
Thus, the set ${\mathcal F}(A)$ will consist of all possible vectors ${\ve b}$
such that the polytope $P(A,{\ve b})$ contains an integer point.

A set $S\subset\R^m$ will be called a  {\em feasible} set if  $S\cap \Z^m \subset{\mathcal F}(A)$.
Results of Aliev and Henk \cite{AH}, Knight \cite{Knight},
Simpson and Tijdeman \cite{ST} and Pleasants, Ray and Simpson \cite{PRS} show that the set
${\mathcal F}(A)$ can be decomposed into the set of all integer points in a certain  feasible (translated) cone and a complementary set
with complex combinatorial structure.

Note that the case $m=1$ corresponds to the celebrated Frobenius problem and has been extensively studied in the literature. We address this problem below.
When $n=m+1$ Pleasants, Ray and Simpson \cite{PRS} obtained a unique
maximal cone whose interior is feasible.
To the best of the authors knowledge the existence of such a maximal cone in the general case is not known.

The location of a feasible cone is given by the {\em diagonal Frobenius number}
defined as follows.
Let ${\ve v}_1,\ldots,{\ve v}_n\in \Z^m$  be the columns  of the matrix $A$ and let
\bea
C=\{\mu_1{\ve v}_1+\cdots+\mu_n{\ve v}_n: \mu_1,\ldots,\mu_n\ge 0\}\,
\eea
be the cone generated by ${\ve v}_1,\ldots,{\ve v}_n$. Let also ${\ve v}:={\ve v}_1+\ldots+{\ve v}_n$. Following Aliev and Henk \cite{AH}, by the {\em diagonal Frobenius number} $\dfrob=\dfrob(A)$ {\em of} $A$ we understand the minimal $s\ge 0$, such that for all ${\ve b}\in (s{\ve v}+C)\cap\Z^m$ the polytope $P(A,{\ve b})$ contains an integer point.
Thus we have the inclusion
\begin{equation*}
(\dfrob(A){\ve v}+C)\cap\Z^m \subset {\mathcal F}(A)\,,
\label{eq:diag}
\end{equation*}
or, in other words, the translated cone $\dfrob(A){\ve v}+C$ is feasible.

The behavior of $\dfrob(A)$ was investigated in Aliev and Henk \cite{AH}. The authors obtained an optimal up to a constant multiplier upper bound
\be \dfrob(A)\le \frac{(n-m)}{2}(n\det(AA^T))^{1/2}\,\label{upper_bound_for_F}\ee
and estimated the expected value of the diagonal Frobenius number.

It is natural to  expect  that the problem (\ref{Knapsack}) is solvable in polynomial time when the right hand side vector ${\ve b}$ belongs to a feasible cone. For such vectors ${\ve b}$ we a priori know that the knapsack polytope contains at least one integer point.
We would like to propose the following conjecture.
\begin{conj}
The problem (\ref{Knapsack}) is solvable in polynomial time for
all instances $(A,{\ve b})$ with
\bea
{\ve b} \in (\dfrob(A){\ve v}+C)\cap\Z^m\,.
\eea
\label{conjecture}
\end{conj}

This question is closely related to algorithmic problems in Section A.1 of Ram\'{\i}rez Alfons\'{\i}n \cite{Alf}.

The first result of the paper gives an estimate for the location of the desired feasible cone  and can be considered as a step towards proving our conjecture.

\begin{theo}

There exists a polynomial time algorithm which, given $(A, {\ve b})$, where $A$ satisfies (\ref{assumption}), ${\ve b}\in \Z^m$ with
\be{\ve b}\in (2^{(n-m)/2-1}p(m,n)(\det(AA^T))^{1/2}{\ve v}+C)\,\label{desired_cone}\ee
and
\bea p(m,n)=2^{-1/2}(n-m)^{1/2}n^{1/2}\,,\eea
finds an integer point in the polytope $P(A, {\ve b})$.

\label{polynomial_algorithm}
\end{theo}

In view of (\ref{upper_bound_for_F}), the affirmative answer to our conjecture would imply that the factor
$2^{(n-m)/2-1}p(m,n)$ in (\ref{desired_cone}) can be replaced by
$\frac{(n-m)n^{1/2}}{2}$, hence
 the exponent $2^{(n-m)/2-1}$ in (\ref{desired_cone}) might be
 redundant.

Our next result shows that the exponent can be removed  for all matrices $A$ with sufficiently large $\det(AA^T)$. This phenomenon is related to the  bounds on the efficiency of the LLL-algorithm and is a consequence of Theorem \ref{main_approximation_theorem} below.
In order to state the result,
let $\gamma_k$ be the $k$-dimensional Hermite constant for which we
refer to \cite[Definition 2.2.5]{Martinet}. Here we just note that by
a result of Blichfeldt  (see, e.g., Gruber and Lekkerkerker \cite{GrLek})
\bea
\gamma_k\le 2\left(\frac{k+2}{\sigma_k}\right)^{2/k}\,,
\eea
where $\sigma_k$ is the volume of the unit $k$-ball;  thus
$\gamma_k=O(k)$.

\begin{theo}
There exists a polynomial time algorithm which, given $(A, {\ve b})$, where $A$ satisfies (\ref{assumption}), ${\ve b}\in \Z^m$ with
\be{\ve b}\in (p(m,n)(\det(AA^T))^{1/2}{\ve v}+C)\label{desired_cone1}\ee
and
\be
\det(AA^T)> \frac{(n-m)2^{2(n-m-2)}\gamma_{n-m}^{n-m}}{n^2}\,,
\label{lower_det}
\ee
finds an integer point in the polytope $P(A, {\ve b})$.

\label{improvement}
\end{theo}

Thus, if the dimension $n$ is concerned,  Theorem \ref{polynomial_algorithm} gives an exponential bound in $n$ for
the location of the desired feasible cone, the affirmative answer to Conjecture \ref{conjecture} would imply the bound  of order $n^{3/2}$ and for large determinants $\det(AA^T)$ we obtained the
bound of order $n$ in Theorem \ref{improvement}.
In view of the size of $\gamma_k$, the lower bound for $\det(AA^T)$ in (\ref{lower_det})
has order $n^{-1} 2^{n\log n+2n}$.

We would also like to mention an interesting consequence of Theorems  \ref{polynomial_algorithm} and
\ref{improvement}. The proof of Lemma 1.1 in Aliev and Henk \cite{AH} immediately implies that for any integer vector ${\ve w}$ in the interior $\interior C$ of the cone $C$
we have
\bea
\left(\frac{\det(AA^T)}{n-m+1}\right)^{1/2}{\ve w}\in ({\ve v}+C)\,.
\eea
It follows then from Theorem  \ref{polynomial_algorithm} that for every integer vector ${\ve b}\in \interior C$
one can find in polynomial time an integer point in the polytope $P(A,\gamma{\ve b})$
for any integer vector $\gamma{\ve b}$ with
\bea \gamma> \frac{2^{(n-m)/2-1}p(m,n)}{n-m+1}\det(AA^T)\,.\eea
Moreover, if we assume (\ref{lower_det}) to hold, then by Theorem \ref{improvement} we can remove the exponential multiplier $2^{(n-m)/2-1}$ from the latter inequality.

Let us now consider the special case $m=1$.
Then $A={\ve a}^T$ with ${\ve a}=(a_1,
a_2, \ldots, a_n)^T\in\Z^n$ and  \eqref{assumption} i) says that
 $\gcd({\ve a}):=\gcd(a_1, a_2,
\ldots, a_n)=1$. Due to the second assumption \eqref{assumption} ii)
we may assume that all entries of ${\ve a}$ are positive. The largest integral value $b$ such that for $A={\ve a}^T$ and ${\ve b}=(b)$ the polytope $P(A,{\ve b})$ contains no integer point is called the {\em Frobenius number} of ${\ve a}$\,,
denoted by $\frob({\ve a})$.
Frobenius numbers naturally appear in the analysis of integer
programming algorithms (see, e.g.,  Aardal and Lenstra
\cite{Aardal_Lenstra}, Hansen and Ryan \cite{Hansen_Ryan}, and Lee,
Onn and Weismantel \cite{Lee_Onn_Weismantel}). The general problem of
finding $\frob({\ve a})$ has been traditionally referred to as the
{\em Frobenius problem}.  This problem is NP-hard (Ram\'{\i}rez
Alfons\'{\i}n \cite{Alf1, Alf}) and integer programming techniques are known to be an effective tool for investigating behavior of the Frobenius numbers, see e.g. Kannan \cite{Kannan}, Eisenbrand and Shmonin \cite{ES} and Beihoffer et al \cite{BHNW}.

Thus, when $m=1$ the answer for the feasibility problem
\be
\begin{array}{l}
\mbox{Given input}\; (A,{\ve b}),\;\mbox{does the polytope}\; P(A,{\ve b})\;\\ \mbox{contain an integer point?}\,
\end{array}
\label{F_Knapsack}
\ee
is affirmative for all instances $({\ve a}^T, b)$ with $b>\frob({\ve a})$. Therefore, it is natural to expect that for $m=1$ the problem (\ref{Knapsack}) can be solved in polynomial time when $b> c$, for some function $c=c({\ve a})$.
 To the best of our knowledge, this conjecture with $c=\frob({\ve a})$ was first stated by Ram\'{\i}rez Alfons\'{\i}n (for related algorithmic questions see Section A.1 in \cite{Alf}).
Note that if the answer to the latter conjecture is affirmative,
then the factor
$2^{(n-1)/2-1}p(1,n)$ in (\ref{desired_cone}) can be replaced by an absolute constant.

Let  $||\cdot||$ denote the Euclidean norm.
In the case $m=1$, Theorems \ref{polynomial_algorithm} and \ref{improvement} deal with input instances $({\ve a}^T, b)$, satisfying the inequalities
$b>2^{(n-1)/2-1}p(1,n)||{\ve a}||\sum_{i=1}^n a_i$ and $b>p(1,n)||{\ve a}||\sum_{i=1}^n a_i$, respectively.
However, in this important special case, one can use slightly refined lower bounds for $b$. The bounds naturally follow from
the geometric approach to the Frobenius problem developed in Kannan \cite{Kannan} and are closely related to
the upper bound obtained in Fukshansky and Robins \cite{Lenny}.

Let ${\ve a}[i]=(a_1,\ldots,a_{i-1},a_{i+1},\ldots,a_N)$. For $m=1$ we obtain the following refinement of Theorems \ref{polynomial_algorithm} and \ref{improvement}.

\begin{theo} Let $\delta>0$. Then the conditions (\ref{desired_cone}) and (\ref{desired_cone1}) in the statements of Theorems \ref{polynomial_algorithm}
and \ref{improvement} can be replaced by
\be b\ge 2^{(n-1)/2-1}(1+\delta)p(1,n)\sum_{i=1}^n ||{\ve a}[i]||a_i\,\label{desired_cone3}\ee
and
\be b\ge (1+\delta)p(1,n)\sum_{i=1}^n ||{\ve a}[i]||a_i\,,\label{desired_cone4}\ee
respectively.
\label{special_polynomial_algorithm}
\end{theo}

The proofs of Theorems \ref{polynomial_algorithm} and \ref{improvement} are based on the classical Babai's nearest point algorithm \cite{Babaika}. The algorithm is searching for a nearby lattice point and is built on the LLL lattice basis reduction (see Section \ref{GSSM}).  The key ingredient of the proof of Theorem \ref{improvement} is the following result.

\begin{theo}
Let $\rho_k=\frac{k2^{2(k-2)\gamma_k^k}}{n^2}$. If $L\subset\Z^n$ is a $k$-dimensional lattice and
${\ve b}_1, {\ve b}_2, \ldots, {\ve b}_k$ is an LLL--reduced
basis of $L$,
then
\be
||{\ve b}_i||\le  \left(1+\frac{\rho_k}{(\det(L))^2}\right )^{1/2}\sqrt{n}\det(L)\,,\;\; i=1,\ldots,k\,.
\label{polynomial_bound_LLL}
\ee
\label{main_approximation_theorem}
\end{theo}

Note that the classical bounds for the lengths of the vectors in an LLL-reduced basis
imply for all $1\le i \le k$ the estimates
\bea
||{\ve b}_i||\le  2^{\frac{k-1}{2}}n^{1/2}\det(L)\,,
\eea
see Lemma \ref{approximation_lemma} below. In (\ref{polynomial_bound_LLL})
we manage to remove the exponential multiplier $2^{(k-1)/2}$ for integer lattices with sufficiently large
determinant.

\section{Integer Knapsacks and Geometry of Numbers}
\label{FL}

Our approach to the problem  is based on Geometry of Numbers for which
we refer to the books \cite{Cassels, peterbible, GrLek}.

By a {\em lattice} we will understand a discrete submodule $L$ of a
finite-dimen\-sional Euclidean space. Here we are mainly interested in
primitive lattices $L\subset \Z^n$, where such a lattice is called
{\em primitive} if $L=\spn_{\R}(L)\cap \Z^n$.

Recall that the Frobenius number $\frob({\ve a})$ is defined only for
integer vectors ${\ve a}=(a_1, a_2, \ldots,a_n)$ with $\gcd({\ve
  a})=1$. This is equivalent to the statement that the
$1$-dimensional lattice $L=\Z\,{\ve a}$, generated by ${\ve a}$ is
primitive. This generalizes easily to  an $m$-dimensional lattice
$L\subset\Z^n$ generated by ${\ve a}_1,\ldots, {\ve a}_m\in\Z^n$. Here the criterion
is that $L$ is primitive if and only if the greatest common divisor of
all $m\times m$-minors is 1. This is an immediate consequence of
Cassels \cite[Lemma 2, Chapter1]{Cassels} or see Schrijver
\cite[Corollary 4.1c]{ASch}.

Hence, by our assumption \eqref{assumption} i), the rows of the matrix
$A$ generate a primitive lattice $L_A$. The determinant of an
$m$-dimensional lattice is the $m$-dimensional volume of the
parallelepiped spanned by the vectors of a basis. Thus in our setting
we have
\begin{equation*}
 \det(L_A) = \sqrt{\det(A\,A^T)}.
\end{equation*}

Now let $A\in\Z^{m\times n}$ be a matrix satisfying the
assumptions \eqref{assumption}. By $V_A$ we will denote the $m$-dimensional subspace of
$\R^n$ spanned by the rows of $A$. The orthogonal complement of $V_A$ in $\R^n$ will be denoted as $V^\bot_A$,
so that
\begin{equation*}
 V_A^\perp=\{{\ve x}\in\R^n : A\,{\ve x}={\ve 0}\}\,.
\end{equation*}
Furthermore, we will use the notation
\bea
L_A^\bot=V_A^\bot\cap\Z^n
\eea
for the integer sublattice contained in $V_A^\bot$.
Observe that (see e. g. \cite{PM})
\begin{equation}
        \det(L^\perp_A)=\det(L_A) = \sqrt{\det(A\,A^T)}. 
\label{eq:det}
\end{equation}

 For a
$k$-dimensional lattice $L$ and an $0$-symmetric convex body $K\subset
\spn_\R L$ the $i$th-successive minimum of $K$ with respect to $L$ is
defined as
\begin{equation*}
 \lambda_i(K,L)=\min\{\lambda>0 : \dim(\lambda\,K\cap L)\geq i\},\quad
 1\leq i\leq k,
\end{equation*}
i.e., it is the smallest factor such that $\lambda\,K$ contains at least $i$
linearly independent  lattice points of $L$.

The Minkowski's celebrated theorem on successive
minima states (cf.~\cite[Theorem 23.1]{peterbible})
\begin{equation}
 \frac{2^k}{k!}\det(L)\leq \vol(K)\,\prod_{i=1}^k\lambda_i(K,L)\leq
 2^k\det(L),
\label{eq:second_minkowski}
\end{equation}
where $\vol(K)$ denotes the volume of $K$.

Let $B$ be the unit ball in $\spn_\R L$.
In the important special case $K=B$ the Minkowski's theorem on successive
minima can be improved (cf.~\cite[\S 18.4, Theorem 3]{GrLek}) to
\begin{equation}
 \det(L)\leq \prod_{i=1}^k \lambda_i(B,L)\leq
\gamma_k^{k/2} \det(L)\,.
\label{eq:second_minkowski_balls}
\end{equation}

\section{Auxiliary results}
\label{GSSM}
\newcommand{\gb}{\hat{{\ve b}}}
\newcommand{\vb}{\ve b}

First, we will prove several lemmas that show a relation between the LLL reduction and successive minima.

For a basis ${\ve b}_1, {\ve b}_2, \ldots, {\ve b}_k$ of a lattice $L$ in $\R^n$ we denote by
$\hat{{\ve b}}_1, \hat{{\ve b}}_2, \ldots, \hat{{\ve b}}_k$ its Gram-Schmidt orthogonalization and by
$\mu_{i,j}$ the corresponding Gram-Schmidt coefficients, that is
\bea
\hat{{\ve b}}_1={\ve b}_1\,,\;\;\hat{{\ve b}}_i={\ve b}_i-\sum_{j=1}^{i-1}\mu_{ij}\hat{{\ve b}}_i\,,\;\;2\le i \le k\,,
\eea
and
\bea
\mu_{ij}= \frac{\langle {\ve b}_i, \hat{{\ve b}}_j\rangle}{||\hat{{\ve b}}_j||^2}\,.
\eea
Put  $\lambda_i=\lambda(B, L)$, where $B$ is the unit ball in $\spn_\R L$. Let us recall the following
technical observation.

\begin{lemma}
We have
\bea
\lambda_i \ge \min_{j=i,i+1,\ldots,k} ||\hat{{\ve b}}_j||\,,\;\; i=1,2,\ldots, k\,.
\eea
\label{lambda_via_min}
\end{lemma}

\begin{proof} The proof can be easily derived from the proof of
  Proposition 1.12 in \cite{LLL}. 
  \end{proof}
Recall that a lattice basis ${\ve b}_1, {\ve b}_2, \ldots, {\ve b}_k$ is {\em LLL--reduced} if
\begin{itemize}
\item[(a)] $|\mu_{ij}|\le \frac{1}{2}$, for $1\le j<i\le k$;
\item[(b)] $\frac{3}{4}||\hat{{\ve b}}_{i-1}||^2 \le ||\hat{{\ve b}}_{i}||^2+\mu^2_{i\, i-1}||\hat{{\ve b}}_{i-1}||^2$, for $2\le i\le k$.
\end{itemize}

The next lemma gives well-known upper bounds for the length the $i$th vector of the
LLL-reduced basis.

\begin{lemma}
Suppose that the basis ${\ve b}_1, {\ve b}_2, \ldots, {\ve b}_k$ is LLL--reduced. Then for $1\le i \le k$ the inequalities
\be
||{\ve b}_i||^2\le 2^{i-1}||\hat{\ve b}_i||^2
\label{GS_via_labbda}
\ee
\be
||{\ve b}_i||^2\le 2^{k-1}\lambda_i^2
\label{basis_via_labbda}
\ee
hold.
\label{second_main_lemma}
\end{lemma}

\begin{proof}
The inequalities (\ref{GS_via_labbda}) and (\ref{basis_via_labbda}) can be easily derived from the proofs of Propositions 1.6 and 1.12
in  \cite{LLL}, respectively.
\end{proof}

The next result gives an upper bound for the lengths of the vectors
in an LLL-reduced basis in terms of the
determinant of the lattice. The bound is based on the classical estimates from Lenstra, Lenstra and Lovasz \cite{LLL} and, consequently, involves the exponential multiplier $2^{(k-1)/2}$.

\begin{lemma}
Let $L\subset \Z^n$ be given by an LLL--reduced basis ${\ve b}_1, {\ve b}_2, \ldots, {\ve b}_k$.
Then
\be
\max_{i=1,\ldots, k}||{\ve b}_i|| \le 2^{\frac{k-1}{2}}n^{1/2}\det(L)\,.
\label{approximation_lemma_inequality}
\ee
\label{approximation_lemma}
\end{lemma}

\begin{proof}

By Proposition 1.12 of Lenstra, Lenstra and Lovasz \cite{LLL} for any choice of linearly independent vectors
${\ve x}_1,\ldots,{\ve x}_{k}\in L$ the inequality
\be
||{\ve b}_i|| \le 2^\frac{k-1}{2}\max\{||{\ve x}_1||, \ldots, ||{\ve x}_{k}||\}\,
\label{LLL-ineq}
\ee
holds.

Put $C^{n}=[-1,1]^n$, i.e., $C^n$ is the $n$-dimensional cube of edge
length $2$ centered at the origin.
By a well-known result of Vaaler \cite{Vaaler}, any $k$-dimensional section of the cube $C^{n}$ has $k$-volume at least $2^k$. In particular  we have
\bea
\vol_{k}(C^{n} \cap \spn_{\R}(L))\ge 2^{k}\,.
\eea

Thus, by the Minkowski theorem on successive minima, applied to the section $C^{n} \cap \spn_{\R}(L)$ and $L$, there exist linearly independent vectors
${\ve x}_1,\ldots,{\ve x}_{k}\in L$ such that
\bea
||{\ve x}_1||_\infty \cdots ||{\ve x}_{k}||_\infty\le \det(L)\,,
\eea
where $||\cdot||_\infty$ denotes the maximum norm.

Since  ${\ve x}_i$ are nontrivial integral vectors we have
\bea
\max\{||{\ve x}_1||_\infty, \ldots, ||{\ve x}_{k}||_\infty\}\le \det(L)\,.
\eea
Combining the latter inequality with (\ref{LLL-ineq}) we obtain the inequality (\ref{approximation_lemma_inequality}).

\end{proof}

The last lemma of this section gives an upper bound for the last successive minimum in terms of the determinant of the
lattice.

\begin{lemma}
\bea
\lambda_k \le \sqrt{n}\det(L)\,.
\eea
\label{lambda_det}
\end{lemma}

\begin{proof}

By the Minkowski theorem on successive minima, applied to the set $C^{n} \cap \spn_{\R}(L)$ and the lattice $L$, and by the already mentioned result of Vaaler \cite{Vaaler}, we have
\bea
\prod_{i=1}^k\lambda_i(C^{n} \cap \spn_{\R}(L), L)\le \det(L)\,.
\eea
Since $L\subset \Z^{n}$, the interior of $C^{n} \cap \spn_{\R}(L)$ does not contain any nonzero point of $L$. This implies
\bea
\lambda_k(C^{n} \cap \spn_{\R}(L), L)\le \det(L)\,,
\eea
so that
\bea
\lambda_k \le \sqrt{n}\det(L)\,.
\eea
\end{proof}

\section{Proof of Theorem \ref{main_approximation_theorem}}


%

If for some $l<k$ we have $||{\ve b}_l||>||{\ve b}_k||$ then, similarly to the arguments below, it can be shown that the inequalities (\ref{polynomial_bound_LLL}) hold. Thus we may assume that ${\ve b}_k$ is the longest vector of the basis ${\ve b}_1,\ldots, {\ve b}_k$. Now assume $||{\ve b}_k||\ge (1+c_k)^{1/2}\sqrt{n}\det(L)$ for some $c_k>0$. Write
\bea
{\ve b}_k=\hat{\ve b}_k + \sum_{j<k} \mu_{kj} \hat{\ve b}_j\,,\;\; |\mu_{kj}|\le \frac{1}{2}\,.
\eea
Hence
\bea
||{\ve b}_k||^2 \le ||\hat{\ve b}_k||^2+ \frac{1}{4}\sum_{j<k} ||\hat{\ve b}_j||^2\,.
\eea
Since $\lambda_k\le \sqrt{n}\det(L)$ by Lemma \ref{lambda_det} and $||\hat{\ve b}_k||^2\le \lambda_k^2$ by Lemma \ref{lambda_via_min},
one concludes
\bea
\frac{1}{4}\sum_{j<k} ||\hat{\ve b}_j||^2\ge c_k \cdot n \cdot (\det(L))^2\,.
\eea

Hence there is an $i\le k-1$ with
\bea
||{\ve b}_i||^2 \ge 4 c_k \frac{n}{k} (\det(L))^2\,.
\eea

For this $i$ one obtains by (\ref{basis_via_labbda})
\bea
\lambda_i^2 
\ge 4 c_k 2^{1-k} \frac{n}{k} (\det(L))^2\,.
\eea

Using Lemma \ref{lambda_via_min} and (\ref{GS_via_labbda}), one obtains $\lambda_k^2\ge ||\hat{\ve b}_k||^2\ge 2^{1-k}||{\ve b}_k||^2\ge 2^{1-k} (1+c_k)n (\det(L))^2
\ge 2^{1-k} n (\det(L))^2$. Therefore

\bea
\prod_{j=1}^k \lambda_j^2 \ge \lambda_i^2 \lambda_k^2\ge c_k \frac{n^2}{k} 2^{2(2-k)}(\det(L))^4\,.
\eea

Finally, Minkowski's second theorem (see (\ref{eq:second_minkowski_balls})) implies

\bea
c_k \le \frac{\rho_k}{(\det(L))^2}\,.
\eea

\section{The Algorithm. Proofs of Theorems \ref{polynomial_algorithm} and \ref{improvement}}
\label{section_main_algorithm}

\subsection{Proof of Theorem \ref{polynomial_algorithm}}
\label{subsection_main_algorithm_1}


We shall now give a high level description of an algorithm that satisfies conditions stated in Theorem \ref{polynomial_algorithm}. First the algorithm constructs an arbitrary integer solution ${\ve u}$ to $A{\ve x}={\ve b}$ and a rational solution ${\ve c}$ to $A{\ve x}={\ve b}$ with large positive coordinates. From this one computes an integer point ${\ve z}$ in $P(A, {\ve b})$ as follows. One finds a close vector ${\ve v}$ to ${\ve u}-{\ve c}$ in the lattice $L^\bot_A$ and considers ${\ve z}:={\ve u}-{\ve v}$. The vector ${\ve z}$ is an integer vector, since ${\ve u}$, ${\ve v}$ are integer. It is a solution to $A{\ve x}={\ve b}$, since ${\ve u}$ is and $A{\ve v}=0$.
Next, observe that ${\ve z}={\ve u}-{\ve v}={\ve c}-({\ve c}-{\ve u})-{\ve v}$, the vector $({\ve c}-{\ve u})-{\ve v}$
is short, and ${\ve c}$ has large coordinates. This will imply ${\ve z}\in P(A, {\ve b})$.

Suppose that

  \be {\ve b}\in(\mu(m,n)(\det(AA^T))^{1/2}{\ve v}+C)\cap\Z^m\,\label{condition for b}\ee
with $\mu(m,n)=2^{(n-m)/2-1} p(m,n)$.
The algorithm is presented below:

\begin{itemize}

\item[\bf Input]: $(A,{\ve b})$ with $A$ and ${\ve b}$ satisfying (\ref{assumption}) and (\ref{condition for b})
respectively;

\item[\bf Output]: ${\ve z}\in P(A, {\ve b})\cap\Z^n$;

\item[Step 1]: Find a basis ${\ve x}_1, \ldots, {\ve x}_{n-m}$ of $L^\bot_A$ and an integer solution ${\ve u}$
of the equation $A{\ve x}={\ve b}$.


\item[Step 2]: Find a point ${\ve c}\in P(A,{\ve b})$ with coordinates
\begin{equation*}
                    c_i\geq \mu(m,n) (\det(AA^T))^{1/2}\,,\;\;1\le i\le n\,.
\end{equation*}

\item[Step 3]: Apply the Babai's algorithm for finding a
  nearby lattice point to the basis ${\ve x}_1, \ldots, {\ve x}_{n-m}$
  and the point ${\ve u}-{\ve c}$. The algorithm
  returns a lattice point ${\ve v}\in L^\bot_A$.
%
%

\item[Step 4]: The output vector ${\ve z}:={\ve u}-{\ve v}$.

\end{itemize}

Let us now show that the algorithm satisfies conditions of Theorem \ref{polynomial_algorithm}.

Step 1  can be performed in polynomial time by Corollary 5.3c of Schrijver \cite{ASch}.

To justify Step 2 we will need the following observation.

\begin{lemma}
Let ${\ve b}\in (t {\ve v}+C)\cap\Z^m$, $t\ge 0$. One can find in polynomial time a point
 ${\ve c}\in P(A,{\ve b})$ with all coordinates $c_i\ge t$.
\label{technical_lemma}
\end{lemma}
\begin{proof}
Since ${\ve b}\in (t {\ve v}+C)$, we have ${\ve b}=\sum_{i=1}^n (t+\delta_i){\ve v}_i$ with $\delta_i\ge 0$.
Therefore the polytope $P_t=\{{\ve x}\in P: x_i\ge t\,,\; 1\le i \le n\}$ is not empty. By Lemma 6.5.1 of Gr\"otschel, Lov\'asz and Schrijver \cite{GLS}, one can find in polynomial time a vertex ${\ve c}$ of the polytope $P_t$. The point ${\ve c}$
clearly satisfies conditions of Lemma \ref{technical_lemma}.
\end{proof}

On account of (\ref{condition for b}), we can apply Lemma \ref{technical_lemma} with a rational number $t>\mu(m,n) (\det(AA^T))^{1/2}$.
Thus we obtain in polynomial time a point ${\ve c}\in P(A, {\ve b})$ with coordinates $c_i$
satisfying
\begin{equation*}
                    c_i\geq \mu(m,n) (\det(AA^T))^{1/2}\,,\;\; 1\le i \le n\,.
\end{equation*}

The algorithm of Babai (see \cite{Babaika}), applied at Step 3, runs in polynomial time as well.
Thus it is enough to show that the output vector ${\ve z}$ is in the polytope $P(A, {\ve b})$.

  Clearly, the polytope $P(A,{\ve b})$ contains a ball centered at ${\ve c}$ with radius
$r\geq \min_{i} c_i$, so that
\be
 r \geq\mu(m,n) (\det(AA^T))^{1/2}.
\label{lower_bound}
\ee

Since $A{\ve u}={\ve b}$ and ${\ve v}\in L^\bot_A$, the output vector ${\ve z}$ satisfies the condition
 $A{\ve z}={\ve b}$. Thus, by (\ref{lower_bound}), it is enough to show that
\be
||{\ve z}-{\ve c}||\le  \mu(m,n)(\det(AA^T))^{1/2}\,.
\label{close_enough}
\ee
The point ${\ve v}$, computed by Babai's algorithm, satisfies
\be
||({\ve u}-{\ve c})-{\ve v}||^2\le (||{\ve b}_1||^2+\cdots+||{\ve b}_{n-m}||^2)/4\,,
\label{Schnorr_bound}
\ee
where ${\ve b}_1, \ldots, {\ve b}_{n-m}$ is an LLL--reduced basis of $L^\bot_A$.
Since  $||{\ve z}-{\ve c}||=||({\ve u}-{\ve c})-{\ve v}||$, by (\ref{Schnorr_bound}) we have
\be
||{\ve z}-{\ve c}||\le \frac{(n-m)^{1/2}}{2}\max_{i=1,\ldots, n-m}||{\ve b}_i||\,.
\label{zc}
\ee
By  Lemma \ref{approximation_lemma} and the choice of $\mu(m,n)$ we  obtain the inequality (\ref{close_enough}).

\subsection{Proof of Theorem \ref{improvement}}

The above algorithm satisfies
the statement of Theorem \ref{improvement} as well. To see this, we only need to  replace $\mu(m,n)=2^{(n-m)/2-1} p(m,n)$ by $\mu(m,n)=p(m,n)$ and to apply
Theorem \ref{main_approximation_theorem}  instead of Lemma \ref{approximation_lemma} in the end of the proof.

\section{Case $m=1$. Proof of Theorem \ref{special_polynomial_algorithm}}

First we will show that the polytope $P({\ve a}^T,b)$ contains a ball of sufficiently large radius
whose center can be computed in polynomial time.

\begin{lemma}
The polytope
$P({\ve a}^T,b)$ contains
an $(n-m)$-dimensional ball centered at a rational point ${\ve c}$ and of radius
\be
    r> \frac{b||{\ve a}||}{(1+\delta)\sum_{i=1}^n ||{\ve a}[i]||a_i}\,.
\label{FF_radius}
\ee
The point ${\ve c}$ can be computed in polynomial time.
\label{inscribed_ball}
\end{lemma}

\begin{proof}

The polytope $P({\ve a}^T,b)$ is the simplex with vertices
${\ve v}_i=(b/a_i){\ve e}_i$, $1\le i \le n$, where ${\ve e}_i$ are
the standard basis vectors. Hence the inner unit normal vectors of the
facets of this simplex (in the hyperplane $\{\ve x\in\R^n : \ve
a^T \ve x =0\}$) are given by
\begin{equation*}
        \ve u_j:=  \frac{||\ve a||}{||\ve a[j]||}\left(\ve e_j - \frac{
              a_j}{||\ve a||^2}\ve a\right), \quad 1\leq j\leq n.
\end{equation*}
Here $\ve e_j$ denotes $j$-th unit vector in $\R^n$, and the facet corresponding to $\ve u_j$ is the convex
hull of all vertices except $(b/a_j)\,\ve e_j$.

 Now let ${\ve c}^*$ be the center of the maximal inscribed ball
in the simplex $P({\ve a}^T,b)$, and let $r^*$ be its radius. Since this
maximal ball touches all facets of the simplex, the radius is   $(n-1)$ times
the ratio of volume to surface area. Standard calculations (see, e.g., Fukshansky and Robins \cite[(19)]{Lenny}, note that the formula contains the redundant factor $1/(n-1)$) gives
\bea
r^*=b\,\frac{||{\ve a}||}{\sum_{i=1}^n ||{\ve a}[i]||a_i}\,.
\eea
Furthermore, we know that  for $1\leq j\leq n$, the vector ${\ve c}^*-r^*\,\ve u_j$ has to lie in
the facet corresponding to $\ve u_j$. Hence the $j$th
coordinate of  ${\ve c}^*-r^*\,\ve u_j$ has to be zero and so we find
\begin{equation*}
 c_j^*=r^*\frac{||\ve a||}{||\ve a[j]||}\left(1-\frac{a_j^2}{||\ve a||^2}\right)=b\,\frac{||{\ve a[j]}||}{\sum_{i=1}^n ||{\ve a}[i]||a_i}.
\end{equation*}

The numbers $c^*_j$ are in general not rational. However we can find in polynomial time a
rational approximation ${\ve c}$ of the vector ${\ve c}^*$ which satisfies the statement of the lemma.

\end{proof}

Suppose that

\be b\ge (1+\delta)\mu(1,n)\sum_{i=1}^n ||{\ve a}[i]||a_i\,,\label{condition for b_1}\ee
where $\mu(m,n)=2^{(n-m)/2-1} p(m,n)$, as in Section \ref{subsection_main_algorithm_1}.
To prove Theorem \ref{special_polynomial_algorithm} we have to find in polynomial time an integer point in $P({\ve a}^T,b)$.

Recall that ${\ve a}[i]=(a_1,\ldots,a_{i-1},a_{i+1},\ldots,a_N)$.
We propose the following modification of the algorithm from Section \ref{section_main_algorithm} for solving this problem.

Steps 1, 3 and 4 remain the same. Step 2 will be modified as follows

\begin{itemize}

\item[Step 2*]: Find a point ${\ve c}$ such that $P({\ve a}^T,b)$ contains
an $(n-m)$-dimensional ball centered at ${\ve c}$ and of radius
\bea
    r> \frac{b||{\ve a}||}{(1+\delta)\sum_{i=1}^n ||{\ve a}[i]||a_i}\,.
\eea
\end{itemize}

Step 2* is justified by Lemma \ref{inscribed_ball}.
To prove correctness of the algorithm, it is now enough to show that the point ${\ve z}$ obtained at Step 4 satisfies
\be
||{\ve z}-{\ve c}||\le  r\,.
\label{close_enough_1}
\ee
Observe that, as in the proof of Theorem \ref{polynomial_algorithm}, the inequality (\ref{zc}) holds.
%
%
Consequently, by Lemma \ref{approximation_lemma} and (\ref{condition for b_1}) we  obtain the inequality (\ref{close_enough_1}).
Thus the condition (\ref{desired_cone}) can be replaced by (\ref{desired_cone3}).

Next, let us replace $\mu(m,n)=2^{(n-m)/2-1} p(m,n)$ by $\mu(m,n)= p(m,n)$.
In this case, by Theorem \ref{main_approximation_theorem} (for simplicity applied with $\rho_k/(\det(L))^2$ replaced by $1$) we  obtain the inequality (\ref{close_enough_1}) as well. Thus the condition (\ref{desired_cone1}) can be replaced by (\ref{desired_cone4}).

\section{Acknowledgement}

The authors are very grateful to anonymous referees for numerous comments and remarks which significantly enhanced the exposition and improved results of this paper and, especially, for suggesting a new proof of Theorem \ref{main_approximation_theorem}.

\end{document}